\documentclass{article}
\usepackage[utf8]{inputenc}
\usepackage[english]{babel}
\usepackage{amsmath}
\usepackage{amssymb}
\usepackage{amsthm}
\usepackage{enumerate}
\usepackage{color,graphicx}
\usepackage{lastpage209}

\def\F{{\mathbb F}}

\def\A{\mathcal A}
\def\B{\mathcal B}
\def\L{\mathcal L}

\def\SS{\mathcal S}

\def\dim{{\rm dim}\,}

\def\R{{\mathbb R}}

\newtheorem{theorem}{Theorem}[section] 
\newtheorem{lemma}[theorem]{Lemma}
\newtheorem{corollary}[theorem]{Corollary}
\newtheorem{proposition}[theorem]{Proposition}
\theoremstyle{definition}
\newtheorem{definition}[theorem]{Definition}
\newtheorem{remark}[theorem]{Remark}

\newtheorem{example}[theorem]{Example}

\title{Algebras of slowly growing length \footnote{The work was financially supported by the grant  RSF 17-11-01124.}}

\author{A. Guterman, D. Kudryavtsev}
\date{}

\newcommand{\Addresses}{
	\bigskip
	\footnotesize
	
	\textsc{Alexander Guterman, }{Faculty of Algebra, Department of
		Mechanics and Mathematics, Lomonosov Moscow State University, Moscow
		119991, Russia; Moscow Center for Fundamental and Applied Mathematics,  Moscow,   119991, Russia; Moscow Center for Continuous Mathematical Education,  Moscow, 119002, Russia
	}\par\nopagebreak
	\textit{E-mail address}:   \texttt{guterman@list.ru}
	
	\medskip

	\textsc{Dmitry Kudryavtsev, }{School of Mathematics, University of Manchester, Manchester M13 9PL, UK; Moscow Center for Fundamental and Applied Mathematics,  Moscow, 119002, Russia}\par\nopagebreak
	\textit{E-mail address}:   \texttt{dmitry.kudryavtsev@postgrad.manchester.ac.uk}
}

\begin{document}
\maketitle

\Addresses

\begin{abstract}
We investigate the class of finite dimensional not necessary associative algebras that have slowly growing length, that is, for any algebra  in this class its length is less than or equal to its dimension. We show that this class is considerably big, in particular, finite dimensional Lie algebras as well as many other important classical finite dimensional algebras belong to this class, for example, Leibniz algebras, Novikov algebras, and Zinbiel algebras. An exact upper bounds for the length of these algebras is proved. To do this we transfer the method of characteristic sequences to non-unital algebras and find certain  polynomial conditions on the algebra elements that guarantee the slow growth of the length function. 

MSC: 15A03,17A99,15A78
\end{abstract}

{\em Keywords}: length of algebras, non-associative algebras, Lie algebras.

\section{Introduction}

Let $\F$ be an arbitrary field. In this paper $\A$ always denotes a finite dimensional not necessarily unital not necessarily associative $\F$-algebra with the operation $(\cdot)$ usually denoted by the concatenation. Let $\SS=\{a_1,\ldots,a_k\}$ be a finite generating set of $\A$. Any product of a finite number of elements from $\SS$ is a {\em word} in $\SS$. The {\em length} of the word $w$, denoted $l(w)$, equals to the  number of letters in the corresponding product. If $\A$ is unital, we consider $1$ as a word in $\SS$ with the {\em length $0$}.
It is worth noting that different choices of brackets provide different words of the same length due to the non-associativity of $\A$.

The set of all words  in $\SS$ with the lengths less than or equal to $i$ is denoted by $\SS^i$, here $i\ge 0$.

Note that similarly to the associative case, $m<n$ implies that $\SS^m \subseteq \SS^n$.

The set $\L_i(\SS) = \langle \SS^i \rangle$  is the linear span  of  the set  $\SS^i$
(the set of all finite linear combinations with coefficients belonging to
$\mathbb{F}$). We write $\L_i$ instead of $\L_i(\SS)$ if $\SS$ is clear from the context. It should be noted that for unital algebras $\L_0(\SS)=\langle
1 \rangle=\mathbb{F}$ for any $\SS$, and for non-unital algebras $\L_0 = \emptyset$.
We denote  $\L(\SS) =\bigcup\limits_{i=0}^\infty \L_i(\SS)$. 

Since the set $\SS$ is generating for
$\A$, we have $\A=\L(\SS)$.

\begin{definition}\label{sys_len} 
	The	{\em length of  a generating set} $\SS$ of a finite-dimensional algebra $\A$ is defined as follows: $l(\SS)=\min\{k\in \mathbb{Z}_+:\L_k(\SS)=\A\}.$ 
\end{definition}

\begin{definition}\label{alg_len} The
	{\em length of an algebra $\A$} is   $l(\A)=\max \{l(\SS): \L(\SS)=\A\}$. 
\end{definition}

The problem of the associative algebra length computation was first discussed in \cite{SpeR59, SpeR60} for the algebra of $3\times3$ matrices in the context of the mechanics of isotropic continua.

It is straightforward to see that the length of an associative algebra is strictly less than its dimension, and this bound is sharp. Namely, one-generated associative algebra of the dimension $d$ has the length $d-1$. The first non-trivial result in this direction is going back to Paz \cite{Paz84}. More results on abstract associative algebras can be found, for example, in \cite{LafMSh,Mar09,Pap97}.
However, in general most of the known results on the length function are just bounds that are not sharp. 
Even the sharp upper bound for the length of the matrix algebra is not known, see~\cite{Mar09}. However, a great deal of work has been done investigating the related notion of length for given generating sets of matrices, see~\cite{GutLMSh,Long1,Long2} and references therein.

Recent results on the lengths of non-associative algebras were obtained in the works   \cite{GutK18,GutK19}. In particular   a strict upper bound on the length of a general non-associative unital algebra is provided. 

\begin{theorem}[{\cite[Theorem 2.7]{GutK18}}]
	Let $\A$ be a unital $\F$-algebra,  $\dim \A = n \ge 2$. Then $l( \A)\le 2^{n-2}$.
\end{theorem}

To prove this and several other results  the  method of characteristic sequences
was
introduced, see  \cite{GutK18,GutK19}.

\begin{definition} \label{CharSeqUn} \cite[Definition 3.1]{GutK18}
	Consider a unital $\F$-algebra $\A$   of the  dimension $\dim \A = n$,
	and its generating set $\SS$. By the {\em characteristic sequence} of $\SS$
	in $\A$ we understand a monotonically non-decreasing sequence of  non-negative integers $(f_1, f_2,\ldots, f_N)$, constructed by the following rules:
	\begin{enumerate}
		\item $f_1 = 0$.
		\item Denoting $s_1=\dim \L_1(\SS)$, we define $f_2=\ldots =f_{s_1}=1$.
		\item Let for some $r>0$, $k>1$ the elements $f_1,\ldots, f_r$ be already defined and the sets $ \L_1(\SS) ,\ldots , \L_{k-1}(\SS) $ are considered. Then we inductively continue the process in the following way. Denote $s_k=\dim \L_k(\SS) - \dim \L_{k-1}(\SS)$. We define $f_{r+1} =\ldots =f_{r+s_k}=k$.  
	\end{enumerate}
\end{definition}

It is proved in \cite[Lemma 3.5]{GutK18} that $N=\dim \A$ and $f_N=l(\SS)$.

The main focus of this paper is the algebras with slowly growing length.

\begin{definition}
We say that a class of algebras has {\em slowly growing length}, if for any representative  $\A$ of this class it holds that $l(\A) \le \dim (\A) $.
\end{definition}

For example, associative algebras are of this type since the sequence $\L_k(\SS)$
growths strictly monotone with $k$, and hence can not have more than  $\dim (\A) -1 $ elements.

The main purpose of our paper is to show that the class of algebras with slowly growing length is rather big, namely, a number of important classes of non-associative finite dimensional algebras have slowly growing length, in particular, Lie algebras and  more general classes such as  Leibniz algebras, Novikov algebras, and Zinbiel algebras are of this type. To proceed we generalized the notion of characteristic sequences introduced in \cite{GutK18} to non-unital algebras. It is straightforward to see that if the characteristic sequence $(m_1,\ldots,m_d)$ of an algebra satisfies the condition $m_{j+1}-m_j \le 1 $ for all $j=1,\ldots, d-1$ then the length of this algebra is slowly growing. We find two combinatorial properties for algebras which ensure the aforesaid condition for the characteristic sequences to be fulfilled.  We call the corresponding algebras {\em sliding} and {\em mixing} due to the nature of these combinatorial properties, and investigate their interrelations. After that we examine these combinatorial properties for the major classes of algebras. We prove that associative and Lie algebras, and moreover, Leibniz algebras, are both sliding and mixing. Novikov algebras are mixing, but in general they are not sliding. Zinbiel algebras are sliding, but in general they are not mixing. However, there are algebras that are  neither sliding nor mixing, but have slowly growing length. We provide an example of such algebras. Finally we discuss  algebras that are neither sliding nor mixing and in general are not algebras with slowly growing length. In particular, Valya and Vinberg
algebras are among them. 
 
Our paper is organized as follows. In Section 2 we transfer the method of characteristic sequences to non-unital algebras. In Section 3 polynomial properties which guarantee a slow growth of length are  introduced and length of corresponding algebras is estimated by means of the characteristic sequences. In Section 4 we examine which important classes of non-associative algebras have slowly growing length.

\section{Characteristic sequences for non-unital algebras}

We begin with several definitions and auxiliary results, inherited from the unital case.
Let $\A$ be an $\F$-algebra of the dimension $\dim \A = n$, $n>2$, and $\SS$ be
a generating set for $\A$. The algebra $\A$ can be either unital or non-unital.

\begin{definition}\label{Def_Irr}
	A word $w$  from a generating set $\SS$ of an algebra $\A$ is   {\em irreducible}, if for each integer $m,\ 0 \leq m<l(w),$ it holds  that $w \notin L_m(\SS)$.
\end{definition}

\begin{lemma}\label{lem_1} \cite[Lemma 2.14]{GutK18}
	Any irreducible word $w$, $l(w)>1$,  is a product of two irreducible words of non-zero lengths.
\end{lemma}
 
To work with algebras which are not necessarily unital, we need to generalize Definition \ref{CharSeqUn} for non-unital case. 

\begin{definition} \label{CharSeqB}
By the {\em characteristic sequence} of $\SS$
	in $\A$ we understand a monotonically non-decreasing sequence of  non-negative integers $(m_1,\ldots, m_N)$, constructed by the following rules:
	\begin{enumerate}
		\item If $s_0 = \dim L_0(\SS) =1$, we set $m_1=0$. Otherwise $s_0=0$.
		\item Denoting $s_1=\dim \L_1(\SS) - \dim L_0(\SS) $, we define $m_{s_0+1}=\ldots =m_{s_0+s_1}=1$.
		\item Let for some $r>0$, $k>1$ the elements $m_1,\ldots, m_r$ be already defined and the sets $ \L_0(\SS) ,\ldots , \L_{k-1}(\SS) $ be considered. Then we inductively continue the process in the following way. Denote $s_k=\dim \L_k(\SS) - \dim \L_{k-1}(\SS)$. We define $m_{r+1} =\ldots =m_{r+s_k}=k$.  
	\end{enumerate}
\end{definition}

For a unital algebra we have in respective notations $f_i = m_{i+1}$. The main difference between these sequences is that for non-unital algebra the characteristic sequence starts with~1, while in the unital case it starts with~0.

\begin{lemma} \label{fr_char} Consider a generating set $\SS$ of an algebra $\A$. There exists a finite series of sets $E_1, \ldots, E_{l(\SS)}$, satisfying the following properties:

1. $E_h \subset E_{h+1}$, $h=1,\ldots,l(\SS)-1$

2. $E_h$ is a basis of $L_h(\SS)$.

3. $E_h$ consists of irreducible words in $\SS$ of lengths $0,\ldots,h$, with exactly $s_j = \dim \L_j(\SS) - \dim \L_{j-1}(\SS)$ words of length $j$ for $j=1,\ldots,h$ and $s_0$ words of length $0$, where $s_0$ is 1 for unital algebra and 0 otherwise.

\end{lemma}

\begin{proof}
	We will construct $E_h$ sequentially using induction on $h$.
	
	The base: $h=1$. Assume that $\A$ is unital. Then  we choose the basis $E_1$ as $\{1\} \cup \SS_0$, where $\SS_0$ is the maximal subset of $\SS$, linearly independent modulo~$\F$. 	
	If $\A$ is non-unital, then we choose the basis $E_1$ as the maximal linearly independent subset of $\SS$. In both cases there are exactly $s_0$ irreducible words of length $0$ and $s_1$ words of length $1$ in $E_1$.

	The step. Assume we have constructed $E_h$ for all $h\le k-1$, $2\le k \le l(\SS)$.
	
	By Definition \ref{Def_Irr} $\L_k(\SS)$ is the linear span of all irreducible words of length less than or equal to~$k$. 	We will construct $E_k$ expanding $E_{k-1}$ by a set $E$ using the following algorithm:
	
	\begin{itemize} 
		\item Let $\{w_1,\ldots,w_r\}$ be the set of all irreducible words of length  $k$. It is   finite since the number of   words of length $k$ is finite. Set $t=1$, $E=\emptyset$.
		\item If $w_t \in \langle E \cup E_{k-1} \rangle$, increase $t$ by 1. Otherwise, expand $E$ with $w_t$, making it $E \cup \{w_t\}$, and increase $t$ by 1.
		\item If $t<r$, return to step 2. If $t=r$, end the algorithm.
	\end{itemize}

	We will show that $E_k = E_{k-1} \cup E$ is the desired set. Note that $E \cap E_{k-1} = \emptyset$.
	
	1. $E_k = E_{k-1} \cup E \supset E_{k-1}$.

	2. $E_k$ is a basis of $\L_k(\SS)$. Firstly, it is linearly independent by construction. Secondly, $\langle E_k  \rangle =\L_k(\SS)$, since every irreducible word of length $l\le k-1$ lies in $\L_{k-1}(\SS) = \langle E_{k-1} \rangle \subset \langle E \cup E_{k-1} \rangle$, and  by construction every irreducible word of length $k$ is in $\langle E \cup E_{k-1} \rangle$. Additionally, $E_k$ being a basis of $\L_k(\SS)$ means that $|E|=s_k$ due to $\dim \L_k(\SS) = \dim \langle E_k  \rangle = |E_{k}| = |E| + |E_{k-1}| = |E| + \dim \langle E_{k-1}  \rangle = |E| + \dim \L_{k-1}(\SS)$.

	3. Since $E$ consists only of irreducible words of length $k$ by construction, $|E|=s_k$ as noted above, and $E_{k-1}$ consists of irreducible words of lengths $0,\ldots,k-1$, with exactly $s_j$ words of length $j$ for $j=0,\ldots,k-1$, we have $E_k$ being composed of  irreducible words of lengths $0,\ldots,k$, with exactly $s_j$ words of length $j$ for $j=0,\ldots,k$.

\end{proof}

The following statements provide the analogs of \cite[Lemma 3.4]{GutK18} in non-unital case.
\begin{corollary} \label{cor_core}
	\begin{enumerate}
		\item For any term $m_h$ of the characteristic sequence of $\SS$ there is an irreducible word in $\L(\SS)$ of the length $m_h$.
		\item If there is an irreducible word in letters from $\SS$ of the length $k$, then $k$ is included into the characteristic sequence of~$\SS$.
	\end{enumerate}
\end{corollary}

\begin{proof}
	Consider set $\mathcal{E}=E_{l(\SS)}$, constructed by Lemma \ref{fr_char}.

	 Item 1. If $m_h$ belongs to the characteristic sequence of $\SS$, then $s_{m_h} \neq 0$ and there is at least one irreducible word of length $m_h$ in the set $\mathcal{E}$.

	Item 2. The existence of an irreducible word of length $k$ guarantees that $s_k= \dim \L_k(\SS) - \dim \L_{k-1}(\SS) > 0$, which means that $k$ is included in the characteristic sequence by definition.
\end{proof}

\begin{lemma} \label{N=n}
The characteristic sequence of $\SS$ contains exactly $\dim \A$ terms. Moreover, for the last term we have $m_N=l(\SS)$.
\end{lemma}

\begin{proof}
	There are $s_0 + s_1 + \ldots + s_{l(\SS)}$ terms in the characteristic sequence since for $j > l(\SS)$ it holds that $s_j = \dim L_j(\SS) - \dim L_{j-1}(\SS) = \dim \A - \dim \A =0$. This sum can be rewritten as $\dim L_1 (\SS) +    (\dim L_2(\SS) - \dim L_{1}) + \ldots + (\dim L_{l(\SS)} - \dim L_{l(\SS)-1} = \dim L_{l(\SS)} = \dim \A$.

	By Corollary \ref{cor_core}, Item 2, there is an element of characteristic sequence of $\SS$ which is equal to $l(\SS)$. This implies $m_N \ge l(\SS)$ as the sequence is non-decreasing. However by Corollary \ref{cor_core}, Item 1 there exists an irreducible word of length $m_N$ which means $m_N \le l(\SS)$. Thus, $m_N=l(\SS)$.
\end{proof}

\section{Mixing and sliding algebras}
In this section we study two properties of multiplication which can guarantee a slow growth of length.

Let $x,y,z$ be variables. To introduce  the following definition we need the special  sets $Q_l$ and $Q_r$ of monomials:
$$Q_l(x,y,z) = \{ x   (z   y), x   (y   z), y   (x   z), y   (z   x), x  y, y  x, x  z, z  x, y  z, z  y,x ,y , z \},$$   
here we consider those monomials of degree three where $z$ is an argument of the first multiplication and the multiplier with $z$ is the second factor of the second multiplication,
$$Q_r(x,y,z) = \{ (x   z)   y, (z   x)   y, (y   z)   x, (z   y)   x, x  y, y  x, x  z, z  x, y  z, z  y,x ,y , z \},$$
here we consider those monomials of degree three where $z$ is an argument of the first multiplication and the multiplier with $z$ is the first factor of the second multiplication.

\begin{definition}\label{def_2}
	Let $\A$ be an $\F$-algebra  such that   at least one of the following statements holds:
	
	1.  $z   (x   y) \in \langle Q_r (x,y,z) \rangle$  for all $  x,y,z \in \A$, if $\A$ is non-unital; $z   (x   y) \in \langle Q_r (x,y,z), 1 \rangle$ for all $  x,y,z \in \A$, if $\A$ is unital.
	
	2.  $(x   y)   z \in \langle Q_l (x,y,z) \rangle$ for all $  x,y,z \in \A$, if $\A$ is non-unital; $(x   y)   z \in \langle Q_l (x,y,z), 1 \rangle$ for all $  x,y,z \in \A$, if $\A$ is unital.

Then we call $\A$ a {\em sliding} algebra. 
\end{definition}

To introduce the next class of algebras we need the monomial set:  $P(x,y,z)=$  $$= Q_l(x,y,z) \cup Q_r(x,y,z) = \left \{ \begin{matrix} (x   z)   y, (z   x)   y, (y z)   x, (z   y)   x,  x   (z   y), x   (y   z), y   (x   z), y   (z   x),\\ x  y,  y  x, x  z, z  x, y  z, z  y,x ,y , z \end{matrix} \right \}, $$
i.e. we consider those monomials of degree 3 that have $z$ inside the brackets.

\begin{definition}\label{def_1}
	Let $\A$ be an $\F$-algebra  such that for all $  x,y,z \in \A$ it holds that $(x   y)  z,\ z    (x   y)  \in  \langle P(x,y,z), 1 \rangle$ if $\A$ is unital, and  $(x   y)  z, \ z    (x   y)  \in  \langle P(x,y,z)  \rangle$ if $\A$ is non-unital.	
	Then we call $\A$ a {\em mixing} algebra. 
\end{definition}

\begin{remark}
	Associative algebras are both mixing and sliding.
\end{remark}

We are going to prove the main properties of mixing and sliding algebras that guarantee that these algebras have slowly growing length. 

Let $\widehat{P}(x,y,z)\subset P(x,y,z)$ be the subset of degree 3 monomials. For a mixing algebra $\A$ let 	
$T_l(x,y,z)  \subseteq \widehat{P}(x,y,z) $, respectively $T_r(x,y,z)  \subseteq \widehat{P}(x,y,z)$, be the set of monomials   that are included with    non-zero coefficients in  at least one of the representations of $(x   y)   z$, respectively  $z    (x   y)$,  as linear combinations of the elements of $P(x,y,z)\cup \{1\}$ or $P(x,y,z)$ in~$\A$.

\begin{lemma}\label{lem_smallstep}
Let $\A$ be  a mixing algebra,  $\SS$ be a generating set of $\A$, and   $M=(m_1,\ldots,m_d)$ be a characteristic sequence of $\SS$. Then   $m_{j+1}-m_j  \le 1$  for all $j=1,\ldots,d-1$.
\end{lemma}
\begin{proof}
	Assume the contrary. Let $\A$ be a mixing algebra, $\SS$ be its generating set, and assume that there exists $j$ such that $1\le j\le d-1$ and the inequality $m_{j+1}-m_j  \le 1$ does not hold. Let $k$ be the smallest index such that $m_{k+1} - m_k \ge 2$. 
	
	Consider a word $w$ of length at least two. It can be uniquely represented as $w=w' \cdot w''$, where $w'$ and $w''$ have non-zero lengths. We denote $s(w) = \min(l(w'),l(w''))$.

	1. Consider an irreducible word $w$ in $\SS$ of length $m_{k+1}$. Then $s(w)>1$. Indeed, if $s(w)=1$ then $w$ is a product of irreducible words of length 1 and $m_{k+1}-1$ by Lemma \ref{lem_1}. Hence by Corollary \ref{cor_core}, Item 2, there is an element  equal to $m_{k+1} -1$ in the characteristic sequence $M$. This is impossible, since $M$ is non-decreasing and $m_k < m_{k+1}-1$ by the assumption.

	 2. Let us choose an irreducible word $w_0$ of length $m_{k+1}$ in $\SS$ such that $s(w_0)$ is the smallest. If there are several such words, we take any one of them. The chosen word can be represented as a product of two irreducible words, $w'_0$ and $w''_0$, such that $l(w'_0)=s(w_0)$. Thus we have the following 2 cases:
	
	Case 1: $w_0=w'_0 \cdot w''_0$. By Item 1 of the proof $s(w_0) \ge 2$. Hence $w'_0 = w'_1 \cdot w'_2$, where $w'_1,w'_2$ are both  irreducible and $l(w'_1)<l(w_0'),\ l(w'_2)<l(w_0')$. The algebra $\A$ is mixing, which means that the irreducible word $w_0 = (w'_1\cdot w'_2) \cdot w''_0 \in  \langle P(w'_1,w'_2,w''_0)  \rangle$, and from this follows that at least one element of $T_l(w'_1,w'_2,w''_0)$ is an irreducible word, as elements of  $P(w'_1,w'_2,w''_0) \setminus \widehat{P}(w'_1,w'_2,w''_0)$  have strictly lesser length than $w_0$. Assume that  $(w'_1 \cdot w''_0) \cdot w'_2$ is irreducible. Then we have $s( (w'_1 \cdot w''_0) \cdot w'_2 ) = l(w'_2) < l(w'_0) = s(w_0)$, which contradicts our choice of $w_0$. For other elements of $T_l$ a similar reasoning holds. Thus the initial assumption is false, i.e. mixing algebra cannot have a generating set with such a characteristic sequence that the difference between neighboring element is greater than~1.
	
	Case 2: $w_0=w''_0 \cdot w'_0$. We obtain the same contradiction similarly, considering an irreducible element of $T_r(w'_1,w'_2,w''_0)$ instead of $T_l(w'_1,w'_2,w''_0)$.
\end{proof}

In the next lemma we need the following sets of monomials. 
Let $\widehat{Q_l}(x,y,z)\subset Q_l(x,y,z)$  and  $\widehat{Q_r}(x,y,z)\subset Q_r(x,y,z)$ be the subsets of degree 3 monomials. For a sliding algebra $\A$ let $S_l(x,y,z)  \subseteq \widehat{Q_l}(x,y,z)$, respectively $S_r(x,y,z)  \subseteq \widehat{Q_r}(x,y,z)$, be the set of monomials   that are included with  non-zero coefficients in  at least one of the representations of $(x   y)   z$, respectively  $z    (x   y)$,  as linear combinations of elements of $Q_l(x,y,z)\cup \{1\}$ or $Q_l(x,y,z)$, respectively  $Q_r(x,y,z)\cup \{1\}$ or $Q_r(x,y,z)$,  in the algebra~$\A$.

\begin{lemma}\label{lem_smallstepZ}
Let $\A$ be  a sliding algebra,  $\SS$ be a generating set of $\A$, and   $M=(m_1,\ldots,m_d)$ be a characteristic sequence of $\SS$ in $\A$. Then   $m_{j+1}-m_j  \le 1$  for all $j=1,\ldots,d-1$.	 
\end{lemma}
\begin{proof}
	Assume the contrary: let $\A$ be a sliding algebra satisfying Item 1 of Definition \ref{def_2}. The other cases can be considered similarly. Let $\SS$ be a generating set of $\A$ such that for its characteristic sequence $M$ the inequality $m_{j+1}-m_j  \le 1$ does not hold for all $j=1,\ldots,d-1$. Let $k$ be the smallest index such that $m_{k+1} - m_k \ge 2$. 

	Consider a word $w$ of length at least two. It can be uniquely represented as $w=w' \cdot w''$, where $w'$ and $w''$ have non-zero lengths. We denote $l_r(w) =l(w'')$.

	Let us choose such an irreducible word $w_0$ of length $m_{k+1}$ in $\SS$ that  $l_r(w)$ is minimal (if there are multiple possible candidates, we can choose one at random). By Lemma \ref{lem_1} $w_0$ is equal to $w'_0 \cdot w''_0$, where terms are irreducible and of lesser length.
	
	1. $l_r(w_0)=l(w''_0)=1$ cannot hold:  this would mean that $l(w'_0)=m_{k+1}-1$ which by Corollary \ref{cor_core}, Item 2 would mean that there is an element of characteristic sequence equal to $m_{k+1} -1$, and that is impossible: $M$ is non-decreasing and $m_k < m_{k+1}-1$ already.
	
	2. If $l(w''_0) >1$, we can represent $w''_0$ as $w''_1 \cdot w''_2$, where both $w''_1, w''_2$ are of positive length and irreducible. The algebra $\A$ is sliding, which means that the irreducible word $w_0= w'_0 \cdot (w''_1 \cdot w''_2) \in \langle Q_r (w''_1,w''_2,w'_0) \rangle$ (or $\langle Q_r (w''_1,w''_2,w'_0),1 \rangle$), and from this follows that at least one element of $S_r(w''_1,w''_2,w'_0)$ is an irreducible word as elements of  $Q_r(w'_1,w'_2,w''_0) \setminus \widehat{Q_r}(w'_1,w'_2,w''_0)$  have strictly lesser length than $w_0$.  Assume it is $(w''_1 \cdot w'_0) \cdot w''_2$. Then we have a contradiction as $l((w''_1 \cdot w'_0) \cdot w''_2)=l(w_0)$, but  $l(w''_2) < l(w''_0)$.  For other elements of $T_r$ a similar observation holds. Thus the initial assumption is false, i.e. sliding algebra cannot have a generating set with such a characteristic sequence that the difference between neighboring element is greater than 1.

\end{proof}

\begin{theorem}\label{th_sg}
	The length of a mixing or a sliding algebra $\A$ of dimension $d \ge 2$ is less than or equal to~$d$.
\end{theorem}
\begin{proof}
	Follows directly from Lemma \ref{lem_smallstep} or Lemma \ref{lem_smallstepZ}: for a generating set $\SS$ of $\A$ with $l(\SS)=l(\A)$ and characteristic sequence $(m_1,\ldots,m_d)$ we have $m_1\le1$ and $l(\SS) =m_d \le m_{d-1} +1 \le \ldots \le m_1 + (d-1) \le d$.
\end{proof}

However, it is not necessary for an algebra $\A$ to be mixing or sliding to satisfy $l(\A) \le \dim (\A)$ as the following example shows.

\begin{example}

Consider an algebra $\A$  over field $\F$ with basis $e_0=1_\F, e_1, \ldots, e_4$ and the following multiplication law:

$$e_1 e_1 = e_2, \ e_2 e_2 =e_3, \ e_1 e_3 = e_4,$$

and other products equal to 0. This is a so-called bare algebra of the sequence $(0,1,2,4,5)$ and its length is equal to 5, as is its dimension, see \cite[Definition 3.3, Theorem 3.15]{GutK19}.

However, it is neither mixing nor sliding. To prove this, consider $$P(e_1,e_1,e_2) = \left \{ \begin{matrix} (e_1   e_2)   e_1, (e_2   e_1)   e_1,  e_1   (e_2   e_1), e_1  (e_1  e_2),\\ e_1  e_1,  e_1  e_2, e_2  e_1, e_1, e_2 \end{matrix} \right \}=$$

$=\{0, e_1,e_2\}$. Since $(e_1 e_1) e_2 =e_3 \not\in \langle  P(e_1,e_1,e_2) \cup \{1\} \rangle$, $\A$ is not mixing. As $P (e_1, e_1, e_2) \supset Q_l(e_1,e_1,e_2)$, this also means that the first property of Definition \ref{def_2} does not hold. To demonstrate that the second property does not hold, note that $(e_1 e_1) e_2 = e_2 (e_1 e_1)$ and $P (e_1, e_1, e_2) \supset Q_r(e_1,e_1,e_2)$, which means $ e_2 (e_1 e_1)\not\in \langle  Q_r(e_1,e_1,e_2) \cup \{1\} \rangle$.

\end{example}

\section{Important classes of non-associative algebras and slowly growing length}

Non-associative algebras are very important in mathematics and its applications, see \cite{book1,book2,book3} and their bibliography. Now we examine standard classes of algebras of slowly growing length. Recall that in this paper all algebras are finite dimensional.

The following lemma is useful in establishing various examples for algebras with polynomial identities.

\begin{lemma}\label{lem_distr}
	Consider a finite-dimensional algebra $\A$ over field $\F$, its basis $\{e_1,\ldots,e_d\}$ and multilinear function $G$ of $k$ arguments such that $G(e_{i_1}, \ldots, e_{i_k}) = 0$ for all $i_t \in \{1,\ldots,d\}$. For all $a_1,\ldots,a_k \in \A$ holds $G(a_1, \ldots, a_k) =0$.
\end{lemma}
\begin{proof}
As $\{e_1,\ldots,e_d\}$ is a basis of $\A$, there exist such $r_{ij} \in \F$ that $a_i = r_{i1}e_1 + \ldots + r_{id}e_d$ for all $i \in {1,\ldots,k}$. We have 

$$G(a_1, \ldots, a_k) = G(r_{11}e_1 + \ldots +r_{1d}e_d, \ldots, r_{k1}e_1 + \ldots +r_{kd}e_d) = $$

$$= \sum_{j_1,\ldots,j_k \in \{1,\ldots,d\} } r_{1 j_1} \ldots r_{k j_k} G( e_{j_1},\ldots,e_{j_k}) = 0.$$
\end{proof}

\subsection{Lie and Leibniz algebras}

\begin{definition}
	An algebra $\A$ is called a {\em Lie algebra} if
	
	1. $xy=-yx$ for all $x,y\in \A$,
	
	2. $(xy)z+(yz)x+(zx)y=0$ for all $x,y,z\in \A$.
\end{definition}

Trivially, Lie algebras are non-unital.

One possible generalization of Lie algebras are Leibniz algebras.

\begin{definition}
	An algebra $\A$ is called a {\em Leibniz algebra} if $(xy)z = x (yz) + (xz) y$ for all  $x,y,z\in \A$.
\end{definition}

An overview of Leibniz algebras can be found in \cite{Feldvoss}.

\begin{proposition}
	Leibniz algebras  are both mixing and sliding. 
\end{proposition}
\begin{proof}
	Required properties of Definitions \ref{def_2} and \ref{def_1} are evident from the definition.
\end{proof}

\begin{corollary}
Leibniz algebras have slowly growing length.
\end{corollary}

Below we provide an example that this bound is sharp for the class of Leibniz algebras. 

\begin{example}
	Consider the algebra $\B_d$ with the basis $x_1,\ldots,x_d$, $d \ge 3$ and the following multiplication law:
		$$x_i x_1 = x_{i+1}, \ i=1,\ldots, d-1,$$
	with the other products being zero. Since $l(\B_d) \le d$ and $l(\B_d) \ge l(\{x_1\}) = d$, we have $l(\B_d) = d$.
	
Let us show that $\B_d$ is indeed a Leibniz algebra. We consider arbitrary elements $u,y,z \in \B_d$ and their representations via basis above. Let $c_y$ and  $c_z$ be the coefficients at $x_1$ of $y$ and $z$, correspondingly. 

Due to the multiplication rules. the coefficient at $x_1$ of $y z$  is equal to zero, which implies that $u (y z) =0$. Meanwhile $(u y) z = (u y) (c_z x_1) = c_z (u y) x_1 = c_z (u (c_y x_1)) x_1 = c_z c_y (u x_1) x_1 =  c_y (u (c_z x_1)) x_1 = c_y (u z) x_1 = (u z) (c_y x_1)  = (u z) y$. Combining these two formulas we achieve  $(uy)z = u (yz) + (uz) y$.
\end{example}

For Lie algebras the above bound can be slightly improved.

\begin{proposition}\label{prop_lie}
	The length of a Lie algebra $\A$ of dimension $d \ge 2$ is not greater than $d-1$.
\end{proposition}
\begin{proof}
	Follows directly from Lemma \ref{lem_smallstep}  and the fact that $\A$ cannot be $1$-generated. Indeed,   $a^2=0$ for any $a\in \A$. Thus 1-generated algebra can not be 2-dimensional. Now for a generating set $\SS$ of $\A$ with $l(\SS)=l(\A)$ and characteristic sequence $(m_1,\ldots, m_d)$ we have $m_1=m_2=1$ and $l(\SS) =m_d \le m_{d-1} +1 \le \ldots \le m_2 + (d-2) = d-1$.
\end{proof}

This bound is sharp as well.

\begin{example}
	Consider so-called filiform Lie algebra $\A_d$ with basis $x_1,\ldots,x_d$, $d \ge 3$ and the following multiplication law:
	
	$$x_1 x_i = x_{i+1} = - x_i x_1, \ i=2,\ldots, d-1,$$
	
	with other products being zero. Since $l(\A_d) \le d-1$ and $l(\A_d) \ge l(\{x_1,x_2\}) = d-1$, we have $l(\A_d) = d-1$.
\end{example}

Lie algebras arise from associative algebras by changing the product $x \cdot y$ into $[x,y] = x  \cdot y - y   \cdot x$, and the following statement provides the connections between these two related algebras.

\begin{proposition}
	Let $\A$ be an associative algebra over a field $\F$ with the multiplication $\cdot$: $\A \times \A \rightarrow \A$, and $\A^{(-)}$ be ts adjoint Lie algebra, i.e.,   $\A^{(-)}=(\A, [,])$, where  $[x,y] = x \cdot y - y \cdot x$ for any $x,y \in \A$. Then  any generating set $\SS$ of $\A^{(-)}$   is a generating set of $\A$   and  $l(\SS) \le l^{Lie}(\SS) $, where $l^{Lie}(\SS)$ is the length of $\SS$ in   $\A^{(-)}$ and $l(\SS)$ is its length in~$\A$. 
\end{proposition}

\begin{proof}
	Both statements follow from the fact that for set $\SS \subset \A$ we have  $\L^{Lie}_n (\SS) \subset \L _n(\SS)$. Here $\L^{Lie}_n (\SS)$ is a linear span of all words of length less than or equal to $n$ in   $\SS$ with respect to the product $[ , ]$. We prove this fact by induction.
	
	The base.   For $n=1$ we have $\L^{Lie}_1 (\SS) = \L _1(\SS)$ as the set of all linear combinations of elements from $\SS$.
	
	The step. Assume that the statement holds for $n=1,\ldots,N-1$. For $n=N$ we have $\L^{Lie}_N (\SS) = \bigcup\limits_{i\in \{1,\ldots, N-1\} }  [\L^{Lie}_{i} (\SS), \L^{Lie}_{N-i} (\SS)]$. Applying induction hypothesis to each component  we have  $\L^{Lie}_N (\SS) \subseteq \bigcup\limits_{i\in \{1,\ldots, N-1\} }  [\L_{i} (\SS), \L_{N-i} (\SS)]$. Then using that $ [\L_{i} (\SS), \L_{N-i} (\SS)]
	\subseteq L_{N}(\SS)$, we have the desired inclusion.
	
	 Since $\SS$ is a generating set of $A^{(-)}$, there exists $n_0$ such that  $\A = \L^{Lie}_{n_0}(\SS) \subset \L_{n_0}(\SS) \subset\A$. Thus,  $\L_{n_0}(\SS) = \A$ and  $l^{Lie}(\SS) \ge l(\SS)$.
\end{proof}

Note that the above proposition does not mean that the length of $\A^{(-)}$ is  greater than or equal to the length of  $\A$. Actually, any mutual behavior of these numerical invariants is possible, as the following examples show.

\begin{example}
 	Consider the algebra  $\A_1=\R^2$ over $\R$ with the addition and the multiplication defined coordinate-wise. Then $\A_1$ is  a unital algebra of the dimension~2. Hence, $l(\A_1)=1$. Also $l^{Lie}(\A^{(-)}_1)= 1$ since a product of any two elements in $\A^{(-)}_1$ equals~$0$.
\end{example}

\begin{example}
	Consider $\A_2=\R^3$ over $\R$ with coordinate-wise addition and multiplication. Then $\A_2$ is a unital algebra of the dimension 3. Hence $l(\A_2) \le 2$. Since  $l(\{(0,1,2)\}) =2$ in $\A_2$ it follows that $l(\A_2)=2$. Meanwhile,   $l(\A_2^{(-)})=1$ since any product in $\A_2^{(-)} $ is equal to~0.
\end{example}

\begin{example}
Consider  $\A_3=M_2(\R)$. Then $l(\A_3)=2$, see, for example, \cite{Paz84}. Let us prove that $l^{Lie}(\A_3^{(-)})=3$. To do this we consider the set $\{  G_1, G_2  \}$, where $G_1 = E_{11} - E_{12}$ and $G_2 = E_{21} + E_{22}$, here $E_{ij}$ is the matrix with 1 in (i,j)-th position and 0 elsewhere. Then we have
$$[G_1, G_2] = (E_{11} - E_{12})(E_{21} + E_{22} ) - (E_{21} + E_{22})(E_{11} - E_{12})=$$ $$= -E_{11}-E_{12} - E_{21} + E_{22} = : G_3,$$
$$[G_3,G_1] =  ( -E_{11}-E_{12} - E_{21} + E_{22})(E_{11} - E_{12})  -  (E_{11} - E_{12})( -E_{11}-E_{12} - E_{21} + E_{22}) =$$ $$ = -E_{11} +3 E_{12} - E_{21} +E_{22} =:G_4 .$$

As $[G_1,G_1]=[G_2,G_2]=0$, $[G_2, G_1] = - [G_1, G_2]$ and $G_1, G_2, G_3$ and $G_4$ are linearly independent, the set $\{  G_1, G_2  \}$ is a generating system of length 3. As $l(\A_3^{(-)}) \le 3 $ by Proposition \ref{prop_lie}, we have  $l(\A_3^{(-)}) = 3$.
\end{example}

\subsection{Novikov algebras}

Another well-known class of non-associative algebras is  the class  of Novikov algebras. Their properties can be found, for example, in~\cite{BdG}.

\begin{definition}\label{def_nov}
	An algebra $\A$ is called a {\em Novikov algebra} if 
	
	1.   $x  (y  z) - (x  y)  z = y  (x  z) - (y  x) z$ for all $ x,y,z \in \A$,
	
	2.  $(x y) z= (x z)  y$ for all	$ x,y,z \in \A$.
\end{definition}

\begin{proposition}
Novikov algebras are mixing and hence they have slowly growing length.
\end{proposition}
\begin{proof}
	Required properties of Definition \ref{def_1} are evident from definition of the class.
\end{proof}

The following example shows that this bound is sharp.

\begin{example}
	Consider algebra $C_d$ with basis $x_1,\ldots,x_d$, $d \ge 3$ and the following multiplication law:
		$$x_1   x_i = x_{i+1}, \ i=1,\ldots ,d-1,$$
	with other products being zero. Since $l(C_d) \le d$ and $l(C_d) \ge l(\{x_1\}) = d$, we have $l(C_d) = d$.
	
	$C_d$ is indeed a Novikov algebra.  To prove this, consider elements $u,y,z \in C_d$ and their representations via basis above. Let $c_u$ and  $c_y$ be coefficients of $u$ and $y$ at $x_1$.

	Coefficient at $x_1$ of $u y$ and $u z$ are zero, which means $(u   y)  z= (u   z)  y = (y u)  z = 0$, and the second property of Definition \ref{def_nov} holds, while the first is reduced to $u (y  z) = y (uz)$. For the latter we have  $u (y  z) = u  ((c_y x_1)  z) = c_y u  (x_1  z) = c_u c_y x_1  (x_1  z) = c_u y  (x_1  z) = y ((c_u x_1)  z) =  y (u  z)$.

\end{example}

Another example demonstrates that Novikov algebras are not necessarily sliding.

\begin{example}

	Consider algebra $C$ over field $\F$ with basis $x_1,x_2,x_3,x_4$ and the following multiplication law:
	$$x_1 x_1 = x_2,\  x_1 x_2 = x_3, \ x_2  x_1 = x_4,$$
	with other products being zero.
	
	$C$ is ia Novikov algebra. To prove this we will check the properties of Definition \ref{def_nov} on basis elements and infer it for other elements by Lemma \ref{lem_distr}. 

	For a triple $u,y,z \in \{x_1,x_2,x_3,x_4\}$ any product of the elements $u,y,z$ in any order is zero if at least one of them is not equal to $x_1$, which means that both properties hold in this case. If $u=y=z = x_1$, they also trivially hold.
	
	However, $C$ is not sliding as $(x_1  x_1) x_1 = x_4$ cannot be represented as a linear combination of  elements of $Q_l(x_1,x_1,x_1)=\{x_1  (x_1 x_1),\ x_1 x_1, x_1 \} = \{x_3, x_2 ,x_1\} $ and vice versa $x_1 (x_1  x_1) = x_3$ cannot be represented as a linear combination of elements of $Q_r(x_1,x_1,x_1)=\{(x_1  x_1)  x_1,\ x_1 x_1, x_1 \} = \{x_4, x_2 ,x_1\} $.
\end{example}

\subsection{Zinbiel algebras}

We also consider Zinbiel algebras, for further information on which we direct the reader to \cite{AKO}.

\begin{definition}
	An algebra $\A$  is called a {\em (right)-Zinbiel algebra} if $x (y  z) =(x  y + y x)  z$ for all $x,y,z \in \A$.
\end{definition}

\begin{proposition}
Zinbiel algebras are sliding and hence they have slowly growing length.
\end{proposition}
\begin{proof}
	Required properties of Definition \ref{def_2} are evident from definition of the class.
\end{proof}

The following example shows that the above bound is sharp.

\begin{example}
	
	Consider the algebra $\mathcal{Z}_d$ over the field $\R$ with the basis $x_1,\ldots,x_d$, $d \ge 3$ and the following multiplication law:
		$$x_i x_j = \frac{j}{i+j} x_{i+j}, \ i,j=1,\ldots, d, \ i+j\le d,$$
	with other products being zero. Since $l(\mathcal{Z}_d) \le d$ and $l(\mathcal{Z}_d) \ge l(\{x_1\}) = d$, we have $l(\mathcal{Z}_d) = d$.
	
	$\mathcal{Z}_d$ is indeed a Zinbiel algebra. To prove this we will demonstrate its defining property on basis elements and infer it for other elements by Lemma \ref{lem_distr}. 

	We have for $i,j,k$ such that $i + j+k \le d$ $$x_i (x_j x_k) =  \frac{k}{j+k} x_i x_{j+k} = \frac{k}{j+k+i} x_{i+j+k} = (x_i x_j + x_j x_i ) x_k ,$$
	
	and for  $i,j,k$ such that  $i + j+k > d$  $$x_i (x_j x_k) = 0 = (x_i x_j + x_j x_i ) x_k.$$
	
\end{example}

Now we demonstrate that Zinbiel algebras are not necessarily mixing.

\begin{example}
	Consider algebra $\mathcal{Z}$ over field $\F$ with basis $x_1,x_2,x_3,x_4,x_5$ and the following multiplication law:
		$$x_1 x_2 = x_4 = - x_2 x_1, \ x_4 x_3 = x_5,$$
	with other products being zero.

	$\mathcal{Z}_d$ is a Zinbiel algebra. To prove this we will demonstrate its defining property on basis elements and infer it for other elements by Lemma \ref{lem_distr}. 

	For a triple $u,y,z \in \{x_1,x_2,x_3,x_4,x_5\}$ any product of $u,y,z$ in any order is zero if at least one of them is equal to $x_4,x_5$, which means that the property holds in this case. If $\{u,y,z\} \subsetneq \{x_1,x_2,x_3\}$, the possible products are zero as well. For the remaining possibilities see the table below (the last column checking $u (y  z) =(u  y + y u)  z$).

	\begin{tabular}{|c|c|c|c|}
		\hline
		$u$ & $y$ & $z$ & Result \\
		\hline
		$x_1$ & $x_2$ & $x_3$ & $0  = (x_4 - x_4) x_3$ \\
		\hline
		$x_2$ & $x_1$ & $x_3$ & $0  = (-x_4 + x_4) x_3$ \\
		\hline
		$x_1$ & $x_3$ & $x_2$ & $0 = (0+0)x_2$ \\
		\hline
		$x_3$ & $x_1$ & $x_2$ & $0 = (0+0)x_2$ \\
		\hline
		$x_2$ & $x_3$ & $x_1$ & $0 = (0+0)x_1$ \\
		\hline
		$x_3$ & $x_2$ & $x_1$ & $ 0 = (0+0)x_1$ \\
		\hline

	\end{tabular}

	However, the algebra is not mixing as $(x_1 x_2) x_3 = x_5$ cannot be represented as a linear combination of elements of $P(x_1,x_2,x_3) =$  $$=\left \{ \begin{matrix} (x_1   x_3)   x_2, (x_3   x_1)   x_2, (x_2 x_3)   x_1, (x_3   x_2)   x_1, \\ x_1   (x_3   x_2), x_1   (x_2   x_3), x_2   (x_1   x_3), x_2   (x_3   x_1),\\ x_1  x_2, x_2 x_1, x_1  x_3, x_3 x_1, x_2  x_3, x_3  x_2,x_1 ,x_2 , x_3 \end{matrix} \right \}=$$ $=\{0,x_1,x_2,x_3,x_4\}$.
\end{example}

\subsection{Some classes of algebras that do not have slowly growing length}

A class   of algebras closely connected with Novikov algebras are Vinberg algebras, also known as right-symmetric algebras (RSA), which are the algebras satisfying just the first one of the two conditions determining Novikov algebras, i.e.
\begin{definition}
	An algebra $\A$ is called a {\em Vinberg algebra} if $(x y) z - x  (y z) = (x  z)  y - x  (z  y)$ for all	$ x,y,z \in \A$.
\end{definition}
An overview of such algebras can be found in~\cite{Burde,Car,Dzhum}.

It can be shown that Vinberg  algebras, which are in general neither mixing nor sliding, do not have slowly growing length universally. Below we present an example of such algebra.

\begin{example}
	Consider a non-unitary algebra $\mathcal{R}$ with basis $e_1, e_2, e_3,e_4$ and the following multiplication table (the operation being concatenation):
	$$e_1 e_1 = e_2, \ e_1 e_2 = e_3, \ e_3 e_2 = e_4.$$
	with the other products being zero. The characteristic sequence of the set $\{e_1\}$ is $1,2,3,5$, while $\mathcal{R}$ belongs to the class of Vinberg algebras. To prove the latter, by Lemma \ref{lem_distr} it is enough to check that for $x,y,z \in \{e_1,e_2,e_3,e_4\}$ it holds that
	$$ (x  y) z - x  (y z) = (x  z)  y - x  (z y).$$

	If either of $x,y,z$ is $e_4$, then every term is obviously zero. After substitution of $e_i$ every term has the same length as words in $\{e_1\}$. This length is greater or equal to 3 (as there are three sub-terms of positive length).

 	For words of lengths $3$ and $5$ consider the table below.
	
	\begin{tabular}{|c|c|c|c|}
		\hline
		$x$ & $y$ & $z$ & Result \\
		\hline
		$e_1$ & $e_1$ & $e_1$ & $0 -e_3 = 0-e_3$ \\
		\hline
		$e_1$ & $e_1$ & $e_3$ & $0 -0 = 0-0$ \\
		\hline
		$e_1$ & $e_3$ & $e_1$ & $0 - 0 = 0-0$ \\
		\hline
		$e_3$ & $e_1$ & $e_1$ & $0 -e_4 = 0-e_4$ \\
		\hline
		$e_2$ & $e_2$ & $e_1$ & $0 -0 = 0-0$ \\
		\hline
		$e_2$ & $e_1$ & $e_2$ & $0 -0 = 0-0$ \\
		\hline
		$e_1$ & $e_2$ & $e_2$ & $e_4 -0 = e_4 - 0$ \\
		\hline
	\end{tabular}

	 Words of length $4$ or $6$ and higher are equal to zero, which means that the desired property holds trivially, and the algebra under consideration is a Vinberg algebra.
\end{example}

\begin{definition}
	An algebra $\A$ is called  a {\em Valya algebra} if 
	
	1.  $xy=-yx$  for all $x,y \in \A$,
	
	2.  $J(x_1 x_2,x_3 x_4, x_5 x_6)=0$,  where $J(x,y,z) = (xy)z+(yz)x+(zx)y$ for all $x_1,x_2,\ldots, x_6 \in \A$
	
\end{definition}

An overview of Valya algebras can be found in~\cite{Tar}.

Universally Valya algebras are neither mixing nor sliding, and they do not necessarily have slow growing length. 

\begin{example}
	Consider algebra $\mathbb{V}$ over a field $\F$ with basis $e_1,e_2,e_3,e_4,e_5,e_6$ and the following multiplication laws:
	
	$$e_1 e_2 = e_3 = - e_2 e_1,\ e_2 e_3 = e_4 = - e_3 e_2,$$
	$$e_3 e_4 = e_5 = -e_4 e_3, \ e_4 e_5 = e_6 = - e_5 e_4,$$
	
	with other products being zero. 

It is a Valya algebra: for the first property multiplication is clearly anti-commutative and for the second it is enough to check it on any six basis elements $e_{i_1}, e_{i_2},e_{i_3}, e_{i_4},e_{i_5}, e_{i_6}$ by Lemma \ref{lem_distr}.

$J(e_{i_1} e_{i_2},e_{i_3} e_{i_4},e_{i_5} e_{i_6}) $ is a sum of three words of similar length in letters $\{e_1,e_2\}$, and this length is at least 6 as each $e_{i_j}$ is a word of positive length.

If this length is other than 8, then every summand is zero as there are no non-zero words in this alphabet of such length.

Otherwise consider the summands in $J(e_{i_1},\ldots,e_{i_6})$. They are represented as products of three words of length at least 2 in $\{e_1,e_2\}$. A non-zero word of length 8, equal to $\pm e_6$ can be represented this was only as a product of $\pm e_3, \pm e_4, \pm e_4$ in correct order (one of  $\pm e_4$ being in the outer product).

However, $J(e_3,e_4,e_4)= 0$. Since $J$ is linear and symmetric by its arguments, all other combinations of $\pm e_3, \pm e_4, \pm e_4$ as arguments of $J$ will result in 0 as well.

The generating set $\{e_1,e_2\}$ has characteristic sequence $(1,1,2,3,5,8)$. It follows that $l(\A)=8 > 6 = \dim \mathbb{V}$.
\end{example}

\bigskip

In the previous sections we discussed classes of algebras with slowly growing lengths. We remark that there are many  algebras such that their length is not bounded by the dimension but is bounded by a certain linear function of the  dimension. Below we   present a certain family of such algebras.
	
\begin{proposition} \label{r-ended}
Let $r\ge 2$ be an integer and $\A_r$ be an algebra satisfying~the property: for all $ x,y_1,\ldots,y_r \in \A_r$ and any product $v=y_1\cdots y_r$ (with any placement of parentheses)   the equality $x v = 0$ holds. Then   $l(\A_r) \le (r-1) \dim \A_r$.
\end{proposition}

\begin{proof}
Consider a generating set $\SS$ of $\A_r$ such that $l(\SS) = l(\A_r)$ and its characteristic sequence $M = (m_1, \ldots, m_d)$, where $d= \dim \A_r$. We 
are going to prove that $m_{j+1} - m_j \le r-1$ for all $1\le j\le d-1$ .

	Assume the contrary. Let $k$ be the smallest index such that $m_{k+1} - m_k \ge r$. 
	
	 Consider a word $w$ of length at least two. It can be uniquely represented as $w=w' \cdot w''$, where $w'$ and $w''$ have non-zero lengths. We denote $s(w) = \min(l(w'),l(w''))$.

Consider an irreducible word $w$ in $\SS$ of length $m_{k+1}$. There are two possibilities.

Case 1: $s(w) \le r-1$. Then $w$ is a product of irreducible words of length $s(w)$ and $m_{k+1}-s(w)$ by Lemma \ref{lem_1}. Hence by Corollary \ref{cor_core}, Item 2, there is an element  equal to $m_{k+1} -s(w)$ in the characteristic sequence $M$. This is impossible, since $M$ is non-decreasing and $m_k < m_{k+1}-s(w)$ by the assumption.

Case 2: $s(w) \ge r$. 
If $ s(w)=l(w'')$ then $l(w'') \ge r$. Hence $w = 0$ by the condition on the products of $(r+1)$ factors in $\A_r$.
Otherwise $s(w)=l(w')$. Then   $l(w'') \ge l(w') \ge r$ and again $w = 0$.
Both of these possibilities contradict the fact that $w$ is irreducible.

Thus, the initial assumption is incorrect and $m_{j+1} - m_j \le r-1$ for all $1\le j\le d-1$. This allows us to conclude that $l(\A_r)  = l(\SS) = m_d \le  m_{d-1} + (r-1) \le \ldots \le m_1 + (r-1) (d-1) < (r-1) d$.
\end{proof}

Let us note that if $r=2$ then  the algebras $\A_r$ are sliding, and therefore, the bound is sharp. If $r>2$ then the resulting bound is not sharp. However, for any $r$ there exist algebras which provide growth of length which is linear in dimension with coefficient~$r-1$.

\begin{example}
Consider an algebra $\mathcal{E}_d$ with the basis $x_1,\ldots,x_d$, $d \ge r \ge 2$ and the following multiplication law:
		$$x_j x_1 = x_{j+1}, \ j=1,\ldots,r-2,$$ 
		$$x_i x_{r-1} = x_{i+1}, \ i=r-1,\ldots, d-1,$$
	with other products being zero. We have $$l(\mathcal{E}_d) \ge l(\{x_1\}) = (r-1) d - (r-2)(r-1) .$$
	
Let us prove that $\mathcal{E}_d$ satisfies the conditions of Proposition~\ref{r-ended}. 

At first, we consider $x,y_1,\ldots,y_r \in \{x_1,\ldots, x_d\}$. A word $x v$,  where $v$ is a product of $y_1,\ldots,y_r$, is indeed zero as $v$ cannot be neither $x_1$ nor $x_{r-1}$. So, the required condition holds for the basis of $\mathcal{E}_d$. Then  by Lemma~\ref{lem_distr} it is satisfied for other elements as well. 
\end{example}


\begin{thebibliography}{0}

\bibitem {AKO} J.~Adashev, A.~Khudoyberdiyev., B.~Omirov, {\em Classifications of some classes of Zinbiel algebras}, J. Gen. Lie Theory Appl. {\bf 4} (2010), 1--10

\bibitem{Burde} D.~Burde, {\em Left-symmetric algebras, or pre-Lie algebras in geometry and physics}, Centr. Eur. J. Math.  {\bf 4} (2006), 323–357.

\bibitem{BdG} D.~Burde, W.~de Graaf, {\em Applicable Algebra in Engineering}, Comm. and Computing, {\bf 24} (2013), 1--15

\bibitem{Car} P.~Cartier, {\em Vinberg algebras, Lie groups and combinatorics}, Quanta of Maths, Clay Math. Proc., {\bf 11} (2010), 107--126

\bibitem {Dzhum} A.~Dzhumadil’daev, {\em Cohomologies and deformations of right-symmetric algebras}, J. of Math. Sci., {\bf 93} (1999), 836--876

\bibitem{book1} A.~Elduque, H. Myung, {\em Mutations of alternative algebras}, Kluwer Academic Publishers, Boston, 1994

\bibitem {Feldvoss} J.~Feldvoss, {\em Leibniz algebras as non-associative algebras}, Nonassoc. Math. and its Appl. {\bf 729} (2019), 115--149

\bibitem{GutK18} A.~Guterman,  D.~Kudryavtsev, {\em Upper bounds for the length of non-associative algebras}, J. of Algebra {\bf 544} (2019) 483-497 
	
\bibitem{GutK19} A.~Guterman,  D.~Kudryavtsev, {\em Characteristic sequences of non-associative algebras}, Comm. in Alg. {\bf 48}:4 (2020), 1713--1725

\bibitem{GutLMSh} A.~Guterman, T.~Laffey, O.~Markova, H.~\v{S}migoc,  A resolution of Paz's conjecture in the presence of a nonderogatory matrix, Linear Algebra Appl.  {\bf 543} (2018)  234--250.
	
\bibitem{LafMSh} T.~Laffey, O.~Markova, H.~\v{S}migoc, {\em The effect of assuming the identity as a generator on the length of the matrix algebra}, Linear Algebra Appl. {\bf 498} (2016)  378--393.

\bibitem{Long1} W.~Longstaff, P.~Rosenthal, {\em On the lengths of irreducible pairs of complex matrices}, Proc. Amer. Math. Soc., {\bf 139}:11 (2011), 3769--3777

\bibitem{Long2} W.~Longstaff, A.~Niemeyer, O.~Panaia {\em On the lengths of pairs of complex matrices of size at most five}, Bull. Austral. Math. Soc., {\bf 73} (2006), 461--472

\bibitem{Mar09} O.~Markova,  {\em Length function and matrix algebras}, J. of Math. Sci. {\bf 193}:5 (2012), 687--768.
	
\bibitem{Pap97}	C.~Pappacena, {\em An upper bound for the length of a finite-dimensional algebra}, J.~Algebra,  {\bf 197} (1997)  535--545.
	
\bibitem{Paz84} A.~Paz, {\em An application of the Cayley--Hamilton theorem to matrix polynomials in severalvariables}, Linear Mult. Algebra,  {\bf 15} (1984) 161--170.

\bibitem{Sagle} A.~Sagle, {\em Malcev algebras}, Trans. Amer. Math. Soc., {\bf 101}: 3 (1961), 426–-458

\bibitem{book2} R.~Schafer,  {\em   An Introduction to Nonassociative Algebras}, Dover Publications, New York, 1995
 	
\bibitem{SpeR59} A.~Spencer, R.~Rivlin, {\em The theory of matrix polynomials and its applications to the mechanics of isotropic continua}, Arch. Ration. Mech. Anal.  {\bf 2} (1959) 309--336.
	
\bibitem{SpeR60} A.~Spencer, R.~Rivlin, {\em Further results in the theory of matrix polynomials}, Arch. Ration. Mech. Anal.  {\bf 4} (1960)  214--230.

\bibitem{Tar} V.~Tarasov, {\em Quantum dissipative system. IV. Analogues of Lie algebras and groups}, Theor. and Math. Phys. {\bf 110}: 2 (1997), 168--178
	
\bibitem{book3} K.~Zhevlakov, A.~Slinko, I.~Shestakov, A.~Shirshov, {\em Rings that are nearly associative}, Academic Press New York, 1982


\end{thebibliography}
\end{document}